\documentclass{article}
\usepackage[noadjust]{cite}
\usepackage[title]{appendix}
\usepackage{url}
\usepackage{fourier}
\usepackage{xcolor}
\usepackage{amsthm}
\usepackage{amsfonts}
\usepackage{amssymb}
\usepackage{amsgen}
\usepackage{amsmath}
\usepackage{amsopn}
\usepackage{verbatim}
\usepackage{xypic}
\usepackage{pgf}
\usepackage{xspace}
\usepackage{multicol}
\usepackage{makeidx}
\usepackage{eepic}
\usepackage{upref}
\usepackage{pgf}
\usepackage{tikz}
\usepackage[normalem]{ulem}
\usepackage{shuffle,yfonts}
\DeclareFontFamily{U}{shuffle}{}
\DeclareFontShape{U}{shuffle}{m}{n}{ <-8>shuffle7 <8->shuffle10}{}

\allowdisplaybreaks

\theoremstyle{theorem}
\newtheorem{theorem}{Theorem}

\newtheorem{lemma}[theorem]{Lemma}
\newtheorem{conjecture}[theorem]{Conjecture}

\theoremstyle{definition}

\newtheorem{problem}[theorem]{Problem}
\newtheorem{remark}[theorem]{Remark}

\begin{document}

\title{Largest Circle Enclosing Exactly $n$ Interior Lattice Points. II}
\date{}

\author{Jianqiang Zhao}

\maketitle

\begin{center}
Department of Mathematics, The Bishop's School, La Jolla, CA 92037, USA

Email Address: zhaoj@ihes.fr
\end{center}

\begin{abstract}
In \cite{Zhao2025Aug}, the author investigated a class of elementary plane geometry problems closely related to the theme of this work. Here, we prove a weaker version of a previous conjecture by demonstrating that there are infinitely many maximally circlable (MAC) numbers -- positive integers $n$ for which there exists a largest circle enclosing exactly $n$ interior lattice points. Furthermore, by extending numerical computations to $n \le 2700$, we identify two counterexamples to a conjecture in \cite{Zhao2025Aug} regarding the symmetry of the largest circle enclosing a strong MAC number (a MAC number $n$ where $n+1$ is non-MAC). We also propose a potential infinite family of strong MAC numbers derived from Pythagorean triples; the existence of this family would imply the infinity of non-MAC numbers, as conjectured by Zhao. Throughout this paper, we provide extensive data characterizing both MAC and strong MAC numbers alongside their corresponding largest enclosing circles.
\end{abstract}

\vspace{10px}


\vspace{10px}

\section{Introduction.}\label{sec:intro}
While counting lattice points inside and on circles remains a classic and highly active problem in analytic number theory (see, e.g., \cite{Hua1942, Huxley} and numerous other works within MSC category 11H31), the specific topic of this note has received surprisingly little attention (though see \cite{Honsberger1973, Schinzel1958, Sierpinski1959, Steinhaus1964}). The present work serves as a natural extension of Problem \textbf{634} found on page 57 of \cite{exeter2017}:

\begin{quote}
What is the radius of the largest circle that you can draw on graph paper that encloses (in its interior)
\begin{alignat*}{4}
&\text{(a) no lattice points?} &&\text{(b) exactly one lattice point?}\\
&\text{(c) exactly two lattice points?}\qquad &&\text{(d) exactly three lattice points?}
\end{alignat*}
\end{quote}

For convenience, throughout this paper an $n$-\emph{circle} means a circle that encloses exactly $n$ interior lattice points.
Denote by $R_n$ to be the the largest radius of such $n$-circles. Then it is easy to find the answers to the above questions:
\begin{center}
(a) $R_0=\sqrt{2}/2$, \quad (b) $R_1=1$,\quad (c) $R_2=\sqrt{5}/2$,\quad and \quad (d) $R_3=5\sqrt{2}/3$.
\end{center}
It is also straight-forward to see that $R_4=\sqrt{10}/2$ and $R_7=5/3$.
The circles in Figure~\ref{fig:R0-4}, some of which are reproduced from \cite[Figure 1]{Zhao2025Aug}, illustrate these cases.

\begin{figure}[h]
\centering
  \includegraphics[scale=0.5]{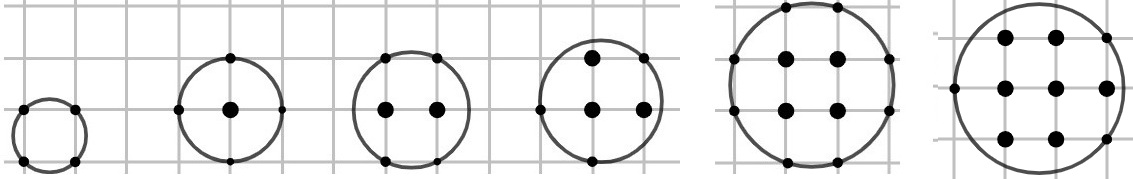}
\caption{The largest circles enclosing zero to four and seven interior lattice points.}
\label{fig:R0-4}
\end{figure}

However, an entirely different phenomenon emerges when one attempts to find $R_5$ and $R_6$; surprisingly, these two numbers do not exist. To state our results precisely, let ${\mathbb N}_0 := {\mathbb N} \cup \{0\}$ denote the set of non-negative integers. We define a number $n \in {\mathbb N}_0$ to be \emph{maximally circlable} (MAC)\footnote{While abbreviated as MC in \cite{Zhao2025Aug}, we adopt MAC here to provide a clear contrast with \emph{minimally circlable} (MIC) numbers, which we investigate in a forthcoming companion paper.} if a largest $n$-circle exists. Otherwise, $n$ is said to be \emph{non-maximally circlable} (non-MAC). Furthermore, a circle is called a \emph{lattice circle} if its circumference contains at least three lattice points, and a \emph{lattice $n$-circle} is an $n$-circle that is also a lattice circle. A \emph{largest $n$-circle} is defined as one possessing the maximal radius among all $n$-circles.

From the examples for $n\le 7$ we see that $0,1,2,3, 4$ and 7 are all MAC numbers while 5 and 6 are non-MAC numbers.
In fact, one can increase the radius of a 5- or 6-circle arbitrarily close to $R_4=\sqrt{10}/2$ but can never make it equal to or exceed $R_4$. See Section~\ref{sec:Proof56} for a quick proof different from that in \cite{Zhao2025Aug}.

It turns out that the same phenomenon for $n=5$ happens quite frequently: there are 433 non-MAC $n$ with $0\le n \le 2700$ according to our computation. However, $n=6$ case is unique in the sense that for each $n\ne 6$ in this range there always exists at least one lattice $n$-circle. We will present more facts for non-MAC numbers in Section~\ref{sec:non-MAC}.

We point out that the method in Section~\ref{sec:Proof56} is intuitive but ad hoc so that it cannot be generalized easily when $n$ is large. On the other hand, the following result which can be shown by the perturbation method is crucial for us to understand the general situation.

\begin{lemma}\label{lem:latticeCircle}  \emph{(\cite[Theorem 4]{Zhao2025Aug})}
Let $n$ be a non-negative integer. Then every largest $n$-circle, if it exists, must be a \textbf{lattice circle}.
\end{lemma}
We caution the reader that for a fixed $n$ such largest circles may not be unique, even modulo \emph{lattice system isometries} (i.e., isometries sending lattice points to lattice points). See Section~\ref{sec:latCircle} for the interesting cases when $n=322$ and $n=941$.

Now, through a quick computer search one can find within a few seconds that there is no lattice 6-circle and the largest lattice 5-circle
\begin{tikzpicture}[scale=0.2,domain=-0.1:3]
\draw[gray] (-0.5,-0.5) grid  (2.5,2.5);
\filldraw (0,1) circle (4pt);
\filldraw (1,0) circle (4pt);
\filldraw (1,2) circle (4pt);
\filldraw (1,1) circle (4pt);
\filldraw (2,1) circle (4pt);
\filldraw (0,0) circle (2pt);
\filldraw (0,2) circle (2pt);
\filldraw (2,0) circle (2pt);
\filldraw (2,2) circle (2pt);
\draw[fill=none](1,1) circle (1.414);
\end{tikzpicture}
has radius $\sqrt{2}<R_4=\sqrt{10}/2$.
Consequently, neither $R_5$ nor $R_6$ exists.

It is conjectured that there are arbitrary long consecutive MAC integer sequences and the set of MAC integers has a positive density in ${\mathbb N}_0$ (see \cite[Conjecture 3, (4) and (6)]{Zhao2025Aug}). In this paper, we prove the following weaker statement of these conjectures.

\begin{theorem}\label{thm:1stA} \emph{(= Theorem~\ref{thm:infinityMAC})}
There are infinitely many MAC numbers.
\end{theorem}

We remark that the sequence of all MAC numbers is named A387044 (and its complement A387045) on the OEIS website \cite{Sloane2025}. Concerning A387045, we have the following conjecture.

\begin{conjecture}\label{conj:nonMAC}
There are infinitely many non-MAC numbers.
\end{conjecture}

It is found that the 433 non-MAC integers in the range $0\le n\le 2700$ spread out quite evenly, which naturally leads to the conjecture that the set of non-MAC integers has a positive density in ${\mathbb N}_0$  (see \cite[Conjecture 3(7)]{Zhao2025Aug}), which is a much stronger statement than Conjecture~\ref{conj:nonMAC}. To study this distribution, we call a MAC number $n$ strong if $n+1$ is a non-MAC number. From numerical evidence, we have the following conjecture relating to Pythagorean triples.

\begin{conjecture}\label{conj:PythagorasStrongMACA} \emph{(cf. Conjecture~\ref{conj:PythagorasStrongMAC})}
Suppose every prime factor $p$ of a positive integer $R$ satisfies $p\equiv 1 \pmod{4}$. Suppose the circle centered at the origin with radius $R$ is an $n_R$-circle. Then $n_R$ is a strong MAC number for infinitely many such $R$'s.
\end{conjecture}

\section{Why 5 and 6 are non-MAC numbers?}\label{sec:Proof56}
In this section, we give a short proof why the largest 5-circle and 6-circle do not exist.

We show that $R_5\ge R_4$ if $R_5$ exists, which leads to a contradiction by Lemma~\ref{lem:latticeCircle} since we have found the largest lattice 5-circle has radius $\sqrt{2}<R_4$. This follows immediately from the next claim.

\medskip
\noindent
\textbf{Claim.} For every positive $\epsilon$, we may perturb the largest 4-circle $O$ as given by the left picture in Figure~\ref{fig:case5-6} by the following two steps so that the perturbed circle with radius $R_4-\epsilon$ contains exactly 5 interior points.
\begin{enumerate}
  \item [(i)] Shrink the circle for $R_4$ by any sufficiently small amount $\varepsilon>0$ (say, $\varepsilon<R_4-1.5\approx 0.08$);
  \item [(ii)] Move the contracted circle along the direction $\langle 1, 3\rangle$ by a distance of $1.2\varepsilon$.
\end{enumerate}

\medskip
To prove this claim, let $O$ be the center of the circle and let $P,Q,R,A_1,\dots,A_5$ be the lattice points on the circumference of $O$ in Figure~\ref{fig:case5-6}. Let $X$ (resp. $U$) be the image of $Q$ (resp. $R$) on the contracted circle pictured below. To verify that $Q$ and $R$ stay outside the contracted circle after steps (i) and (ii), we zoom in at these two points separately and let $Y$ (resp. $V$) be the point on the contracted circle after step (i) such that $\overrightarrow{YQ}$ (resp. $\overrightarrow{VR}$) points to the direction of  $\langle 1, 3\rangle$. Note that such $Y$ and $V$ exist for all $\varepsilon<R_4-1.5$. One deduces quickly that $\alpha=\tan^{-1}3- \tan^{-1}(1/3)=\cos^{-1} (3/5)$ and $\beta=2\tan^{-1}(1/3)$. Hence, $|\overrightarrow{YQ}|>\varepsilon/\cos\alpha=5\varepsilon/3>1.2\varepsilon$ and $|\overrightarrow{VR}|>\varepsilon/\cos\beta=5\varepsilon/4>1.2\varepsilon$. This show that both $Q$ and $R$ remain outside the contracted circle after steps (i) and (ii).

\begin{figure}[h]
\centering
  \includegraphics[scale=0.6]{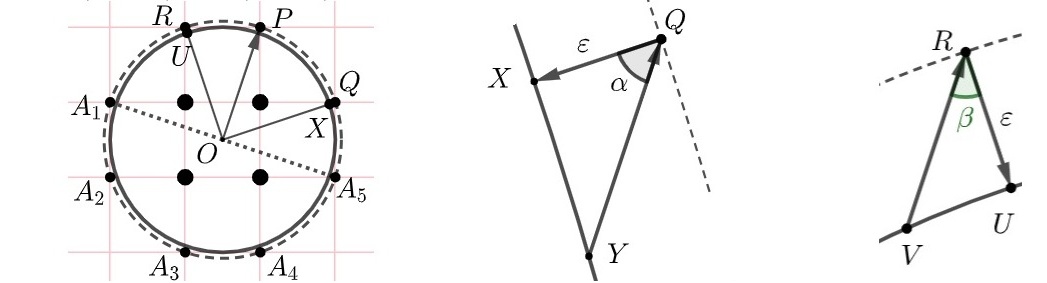}
\caption{Left: Shrink radius by $\varepsilon$. \hskip1cm Middle: Zoom in at $Q$.\hskip1cm Right: Zoom in at $R$.}
\label{fig:case5-6}
\end{figure}

It is also obvious that all $A_j$'s ($j=1,\dots,5$) remain outside the contracted circle after steps (i) and (ii) if we notice that $OP\perp A_1A_5$. The upshot is that only $P$ and the original four interior lattice points are inside the contracted circle after steps (i) and (ii). This completes the proof of the claim.

The above argument implies that $R_5\ge R_4$ if $R_5$ exists since $\varepsilon$ can be arbitrarily small so that the radius of the  contracted circle can be arbitrarily close to $R_4$. One can prove that $R_6\ge R_4$ if $R_6$ exists by similar steps (in particular, one should change the direction in step (ii) to the upward direction). On the other hand, we have found the largest lattice 5-circle has radius $\sqrt{2}<R_4$ and there is no lattice 6-circle. By Lemma~\ref{lem:latticeCircle}, we see that neither $R_5$ nor $R_6$ can exist, in other words, both 5 and 6 are non-MAC numbers.

\section{MAC numbers.}
The key idea used to prove the infinitude of MAC numbers is a detailed analysis of the lattice $n$-circles.

\subsection{Lattice equivalent circles.}
For convenience, we define the \emph{fundamental triangle} $\Delta$ to be the one with vertices at $(0,0)$, $(1/2,0)$, and $(1/2,1/2)$, including its boundary. We say two circles are \emph{lattice-equivalent} if one can be transformed to the other by lattice system isometries, i.e, isometries which send every lattice point to a lattice point. The proof of the following lemma is straight-forward and is thus omitted.

\begin{lemma}\label{lem:inDelta}
There is a unique circle with its center on $\Delta$ in every lattice equivalence class of circles.
\end{lemma}

In the rest of the paper, we will only consider the circles centered on $\Delta$ if we do not mention otherwise.

\subsection{Criterion of MAC and non-MAC numbers.}
For a MAC integer $n$, let $R_n$ be the largest radius of the $n$-circles. For a non-MAC $n$, let $R_n$ be the \emph{least upper bound of the radii} (LUBOR) of all $n$-circles.

The following results are contained in \cite[Main Theorem]{Zhao2025Aug}.

\begin{lemma}\label{lem:main}
Let $\{R_n: n\ge 0\}$ be the positive sequence of MAC-radii (for MAC integers $n$) and LUBORs (for non-MAC integers $n$).
Then we have:
\begin{enumerate}
  \item[(a)] $\{R_n: n\ge 0\}$ is a non-decreasing sequence.

  \item[(b)] Suppose $k<\ell$ are two consecutive MAC integers and $n$ is non-MAC such that $k<n<\ell$. Then $R_n=R_k$, namely, $R_k$ is the LUBOR for all such non-MAC $n$ immediately following $k$.

  \item[(c)] Suppose $k<n$, $k$ is a MAC integer and all the integers between $k$ and $n$ (exclusive) are non-MAC. Then the radius of the largest lattice $n$-circle is \textbf{strictly} smaller than $R_k$ if and only if $n$ is non-MAC.
\end{enumerate}
\end{lemma}

For convenience of the reader, we briefly recall the idea of proof of Lemma~\ref{lem:main}.
Part (a) can be proved by the perturbation method we used when showing $R_5\ge R_4$ in Section 1. For part (b), let $n=k+j$ ($1\le j<\ell-k$). Then an induction on $j$ yields a proof. The key ideas include the following three facts: (i) if $n$ is non-MAC then there is a sequence of circles (centered on $\Delta$) enclosing the same set $S$ of $n$ lattice points such that the radius of the limit circle $\mathcal O$ is the LUBOR $R_n$; (ii) $\mathcal O$ is a lattice circle; and (iii) $\mathcal O$ cannot contain any points in its interior other than $S$ and, moreover, $\mathcal O$ must lose some $S$-points to its boundary. Part (c) then follows from (a), (b) and Lemma~\ref{lem:latticeCircle}. We refer the interested reader to \cite{Zhao2025Aug} for a complete proof of Lemma~\ref{lem:main}.

\subsection{Infinitude of MAC numbers.}
We are now ready to prove one of the main results of this paper.

\begin{theorem}\label{thm:infinityMAC}
There are infinitely many MAC numbers.
\end{theorem}

\begin{proof}
For each $n\ge 0$, let $\rho_n$ be the radius of the largest {\bf lattice} $n$-circle. Set $\rho_n=0$ if such a lattice circle does not exist (e.g., $n=6$).

First, it is easy to see that $n$ is MAC if and only if $\rho_n=R_n$. Indeed, if $n$ is MAC then the largest $n$-circle must be a lattice circle by Lemma~\ref{lem:latticeCircle}, which implies that $\rho_n=R_n$. On the other hand, if $\rho_n=R_n$ then the circle with the corresponding $\rho_n$ is a largest $n$-circle and therefore $n$ is MAC.

Second, for each fixed $M>0$, there are only finitely many lattice circles centered on the fundamental triangle $\Delta$ such that their radii are no more than $M$. By Lemma~\ref{lem:inDelta} we see that
\begin{equation}\label{equ:supRho}
\limsup_{n\to \infty}\rho_n=\infty.
\end{equation}

Now we may determine the sequence of all MAC numbers by the following process. Take $n_1=0$. If we have already found $n_1<\dots<n_\ell$ such that these numbers are the only MAC numbers no more than $n_\ell$, then we set $n_{\ell+1}=\min\{n> n_\ell:\rho_n\ge \rho_{n_\ell}\}$ which exists by \eqref{equ:supRho}. We now show $n_{\ell+1}$ is MAC and all numbers $k$ with $n_\ell<k<n_{\ell+1}$ are non-MAC.
Indeed, $\rho_k<\rho_{n_\ell}$ for $n_\ell<k<n_{\ell+1}$ by definition of $n_{\ell+1}$. Hence, $k$ is non-MAC by Lemma \ref{lem:main}(a). This in turn implies that $n_{\ell+1}$ is MAC by Lemma \ref{lem:main}(c).

This completes the proof of Theorem \ref{thm:infinityMAC}.
\end{proof}

It is conjectured \cite[Conjecture 3, (4) and (6)]{Zhao2025Aug} that there are arbitrarily long consecutive MAC integer sequences and the density of MAC numbers in ${\mathbb N}_0$ is greater than 80\%. Our expanded search for $n\le 2700$ provides further confirmation of these. In Table~\ref{tbl:MACdistrution} we collect the distribution data of MAC numbers in every 300 long interval. Note the MAC numbers should be evenly distributed with density $> 80\%$. We also found in this range the longest consecutive MAC numbers are from 2026 to 2081, totalling 56 numbers.

\begin{table}[h]
\caption{MAC number distribution for $0\le n\le 2700$.}
\centering\begin{tabular}{|c|c|c|c|c|c|c|c|c|c|}
\hline
$\lfloor n/300 \rfloor$  &   0 &  1  & 2  & 3 & 4 & 5 & 6 & 7 & 8 \\
\hline
$\sharp\{n|$  $n$ is MAC$\}$ & 243& 251& 246& 253& 250& 260& 255& 256 & 253  \\
\hline
\end{tabular}
\label{tbl:MACdistrution}
\end{table}

\subsection{An improved algorithm to determine MAC numbers.}
We point out that the proof of Theorem~\ref{thm:infinityMAC} above is constructive in nature. It provides a precise procedure to determine whether a given $n$ is MAC or not.
For example, after discovering $n=889$ is a MAC number with $\rho_{889}=R_{889}=17$, we can find with the aid of a computer that
\begin{align*}
     \rho_{890}=&\,\tfrac{\sqrt{461761}}{40}\approx 16.988,\quad
     \rho_{891}=\tfrac{13\sqrt{19193}}{106}\approx 16.990,\\
     \rho_{892}=&\,\tfrac{\sqrt{242905}}{29}\approx 16.994,\quad
     \rho_{893}=\tfrac{\sqrt{392896610}}{1166}\approx 16.999.
\end{align*}
but $ \rho_{894}=\tfrac{17\sqrt{289445}}{538}>17$.
Hence, 894 is a strong MAC number while 890, 891, 892, and 893 are all non-MAC with their LUBORs $R_{890}=\cdots=R_{893}=17$.

However, the procedure above is not the most efficient to detect if $n$ is MAC or not for large $n$ because it requires us to compute the exact values of $\rho_k$ for all $0\le k\le n$. We now show that, in fact, it suffices to compute a much smaller number (about $\sqrt{2\pi n}$, see Remark~\ref{rem:smallSetRadii}) of radii $\rho_k$'s. We first recall a lemma.

\begin{lemma}\label{lem:Rbound}  \emph{(\cite[Theorem 9]{Zhao2025Aug})}
Let $n$ be a non-negative integer. Suppose $R_n$ is the MAC-radius or the LUBOR of $n$. Then $$\sqrt{n/\pi}<R_n<\sqrt{2}+\sqrt{n/\pi}.$$
\end{lemma}

\begin{theorem}\label{thm:effAlgorithm}
Let $n$ be a positive integer and $k=\lfloor (\sqrt{n}-\sqrt{2\pi})^2 \rfloor$. If
$ \rho_n\ge \rho_l$ for all $k\le l<n$ then $n$ is a MAC number. Otherwise, $n$ is non-MAC.
\end{theorem}

\begin{proof} If $\rho_n<\rho_l$ for some $k\le l<n$ then clearly it is non-MAC by Lemma~\ref{lem:main}(c). Indeed,
assume
\begin{equation*}
   \lambda=\max\{l: k\le l<n, \rho_n<\rho_l\}.
\end{equation*}
If we take $m$ to be the largest MAC number such that $m\le \lambda$ then $\rho_m>\rho_a$ for all $m<a\le n$ by Lemma~\ref{lem:main}(c).
This implies $n$ is non-MAC.

We now assume $\rho_n\ge \rho_l$ for all $k\le l<n$.

Case (1). If $k$ is MAC, then the non-decreasing subsequence of $\{\rho_k,\dots,\rho_n\}$ starting with $\rho_k$ must end with $\rho_n$, which must consist of only MAC numbers by Lemma~\ref{lem:main}(c). Thus, $n$ is a MAC number.

Case (2). If $k$ is non-MAC, then we let $m$ be the largest MAC number smaller than $k$. Since $\rho_m=R_k$ is LUBOR of $k$, by Lemma~\ref{lem:Rbound} we see that $\rho_m<\sqrt{k/\pi}+\sqrt{2}$. By the definition of $k$, we obtain
$$
\rho_m<(\sqrt{n}-\sqrt{2\pi})/\sqrt{\pi}+\sqrt{2}=\sqrt{n/\pi}<\rho_n
$$
by Lemma~\ref{lem:Rbound} again. The same argument as in Case (1) applies here with $k$ replaced by $m$, which shows that $n$ is a MAC number.
\end{proof}

\begin{remark}\label{rem:smallSetRadii}
We point out that the number of radii we need to compute in Theorem~\ref{thm:effAlgorithm} is $n-k+1\approx \sqrt{2\pi n}-2\pi+1$.
It would be very interesting to find a more efficient algorithm better than computing $O(n^{0.5})$ radii to determine if $n$ is MAC or not.
\end{remark}

\section{Lattice circles.}\label{sec:latCircle}
We have seen that lattice circles play the key roles in the proof of Theorem~\ref{thm:infinityMAC}. In this section, we will consider a few interesting problems concerning only lattice circles. For convenience, by a \emph{decoration} of a circle we mean a lattice point on the circumference of the circle. When we consider a circle together with its decorations we call it a \emph{decorated circle}.

\subsection{Uniqueness of the largest MAC $n$-circles.}
Since a largest $n$-circle (if it exists) must be a lattice circle by Lemma~\ref{lem:latticeCircle}, one may wonder if such a circle is unique up to lattice-equivalence. It turns out that this holds true for all MAC numbers $n\le 321$, however, this fails for the MAC number $322$ since we have the following two lattice non-equivalent largest 322-circles with the same radius $R_{322}=5\sqrt{4810}/34$: $M_{322,1}$ (resp. $M_{322,2}$) is centered at $(1/2,3/34)$ (resp. $(9/34,9/34)$ ) containing the following decorations
\begin{align*}
& M_{322,1}: (-7, 7), (-1, -10), (2, -10), (8, 7);\\
& M_{322,2}: (-9, -4), (-5, 9), (-4, -9), (9, -5).
\end{align*}
The non-lattice-equivalence between $M_{322,1}$ and $M_{322,2}$ can also been seen from their decorations: both are trapezoids but they are incongruent. Such non-uniqueness happens also for $n=941$ with two MAC circles: $M_{941,1}$ centered at $(51/134,29/134)$ and $M_{941,2}$ centered at $(57/134,27/134)$, and their decorations are given by
\begin{align*}
& M_{941,1}: (-16,6),(-8,-15),(14,11);\\
& M_{941,2}: (-4,17),(-2,-17),(17,-5).
\end{align*}
These two numbers are the only MAC numbers with more than one lattice non-equivalent MAC $n$-circles for all $n\le 2700$.
However, the number of decorations is the same in both examples. We thus pose the following question.

\begin{problem}\label{prob:decor}
Is there a MAC number $n$ which has two MAC $n$-circles with different number of decorations?
\end{problem}

\subsection{Different decorations on lattice circles with the same radius.}
If we remove the ``largest'' condition then we find two $230$-lattice circles of the same radius that pass through different number of lattice points. Let $C_1$ (resp. $C_2$) be the circle centered at $(1/14,1/14)$ (resp. $(1/2,3/14)$) with radius $85\sqrt{2}/14\approx 8.586$. Then, there are five decoration on $C_1$ while there are only four on $C_2$:
\begin{align*}
& C_1: (-6, -6), (3,-8),(7,-5),(-5,7),(-8,3),\\
& C_2: (-8, -1), (-2, -8), (3, -8), (9, -1).
\end{align*}
Nevertheless, neither circles is the largest 230-circle which is given by the one centered at $(163/334,119/334)$ with radius $145 \sqrt{10988266}/55778\approx 8.617$ with only three decorations: $(-5, 7), (-4, -7)$ and $(7, 6)$.

\subsection{Different number of lattice points inside lattice circles with the same radius.}
It is conjectured (see \cite[Conjecture 3(3)]{Zhao2025Aug}) that the sequence of all MAC-radii $\{R_{n_\ell}\}_{\ell\ge 1}$ is strictly increasing, where $\{n_\ell\}_{\ell\ge 1}$ is the set of all MAC numbers. So far we only know this sequence is non-decreasing by Lemma~\ref{lem:main}(a).
But for general lattice circles, we find that the following two have the same radius $\sqrt{3770}/14$ while enclosing different number of interior lattice points: $C_3$ (resp. $M_{56}$) is centered at $(3/14,3/14)$ (resp. $(1/2,5/14)$). Then $C_3$ (resp. $M_{56}$) encloses exactly 60 (resp. 56) interior lattice points. Note that $M_{56}$ is the largest 56-circle while $M_{60}$ has radius $\sqrt{82}/2$ and is pictured as the first circle in Figure~\ref{fig:MACsymAbundant}.

\section{Strong MAC numbers.}
If $n$ is a MAC number but $n+1$ is not then we say $n$ is a \emph{strong MAC number}. The number of consecutive non-MAC numbers following a strong MAC $n$, denoted by $I_n$, is called its \emph{impacting index}. There are some quite interesting conjectural statements contained in \cite[Conjecture 3]{Zhao2025Aug} concerning the strong MAC numbers. In this section, we will disprove one of them and partially strengthen and confirm some of the others.

\subsection{Impacting indices and number of decorations.}
Frequently, the number of decorations is large if $I_n$ is large. Furthermore, the decorations of the largest circles of the strong MAC numbers often show a lot of apparent symmetries. The following conjecture, a stronger version of \cite[Conjecture 3, (9) and (13)]{Zhao2025Aug}, is supported by the evidence in the range $0\le n\le 2700$. For any MAC number $n$, denote by $M_n$ a largest $n$-circle with its center in the fundamental triangle $\Delta$ and call it an \emph{MAC circle} of $n$.

\begin{conjecture}\label{conj:DecorStrongMAC}
Let $n$ be any strong MAC number. Then $M_n$ has at least $2I_n+4$ decorations.
\end{conjecture}

We now show a result which is weaker than Conjecture \ref{conj:DecorStrongMAC} but strong enough to imply \cite[Conjecture 3, (9)]{Zhao2025Aug} if $I_n\ge 2$, which says that $M_n$ has at least six decorations if $n$ is a strong MAC number $n$.

\begin{theorem}\label{thm:DecorStrongMAC}
For each strong MAC number $n$ with impacting index $I_n$, its MAC circle $M_n$ has at least $2I_n+3$ decorations.
\end{theorem}

\begin{proof}
Suppose there are at most $2I_n+2$ decorations on $M_n$. Draw the diameter $\ell$ of $M_n$ such that it passes through one of the decorations. Then, on one side of $\ell$ there are at most $I_n$ decorations not lying on $\ell$. By moving the center of the circle towards these points along the direction perpendicular to $\ell$ with an infinitesimal distance we see that there is a circle of radius $R_n$ that encloses exactly $n+I_n$ points. By discreteness of lattice points, we can further increase the radius of this circle slightly while keeping the number of interior points unchanged. This contradicts to the condition that $n+I_n$ is non-MAC and that $R_n$ is its LUBOR. This contradiction implies that there are at least $2I_n+3$ decorations on $M_n$.
\end{proof}

\begin{figure}[h]
\centering
  \includegraphics[scale=0.6]{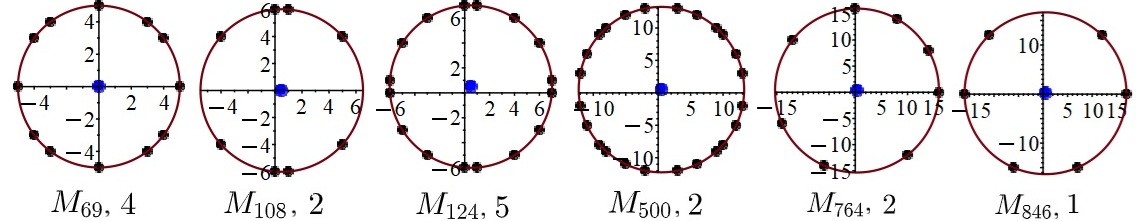}
\caption{The first five decorated MAC circles shown here (with their impacting indices) correct the errors on page 12 of \cite{Zhao2025Aug}. The rightmost decorated circle $M_{846}$ is different from that in \cite{Zhao2025Aug} because we have used a different fundamental triangle $\Delta$ in this paper.}
\label{fig:correction}
\end{figure}

We notice from Section~\ref{sec:intro} that $4$ is strong MAC while its MAC circle $M_4$ in Figure \ref{fig:R0-4} having eight decorations with $I_4=2$. Hence, the lower bound in Conjecture \ref{conj:DecorStrongMAC} cannot be improved in general. The decorated MAC circles in Figure~\ref{fig:correction} provide some more examples.

We also find the following conjecture holds for all strong MAC numbers $n\le 2700$, which would imply Conjecture~\ref{conj:DecorStrongMAC} by Theorem~\ref{thm:DecorStrongMAC}.

\begin{conjecture}\label{conj:DecorStrongMACEven}
For any strong MAC number $n$, the number of decorations on its MAC circle $M_n$ must be even.
\end{conjecture}

Note that our numerical computation ends at $n=2700$ and we found 2700 is MAC but not strong by Theorem~\ref{thm:DecorStrongMAC} since $M_{2700}$ has only 4 decorations. The following table shows the distribution of the impacting indices. We define the impacting index $I_n=-1$ if $n$ is non-MAC and $I_n=0$ if $n$ is MAC but not strong.
\begin{table}[h]
\caption{Distribution of impacting indices $I_n$, for all $0\le n\le 2700$.}
\centering\begin{tabular}{|c|c|c|c|c|c|c|c|c|c|c|c|c|c|}
\hline
$I$ &   $-1$ & 0 & 1 & 2 & 3 & 4 &5 &6 &7 &8 &9 &10&$\ge1$\\
\hline
$\sharp\{n:I_n=I\}$ & 433& 2100& 61& 60& 5& 15& 8& 15& 0& 2& 0& 1 &167\\
\hline
\end{tabular}
\label{tbl:Distribution}
\end{table}
It is conjectured in \cite[Conjecture 3(5)]{Zhao2025Aug} that there are arbitrarily long sequence of non-MAC numbers, namely, $I_n$ is unbounded. We suspect a more stronger statement holds.

\begin{conjecture} \label{conj:infiniteI_n}
The set of impacting indices $\{I_n: n\in{\mathbb N}, \ n \ \text{is strong MAC}\}={\mathbb N}$.
\end{conjecture}

\subsection{Symmetry of decorations.}
For strong MAC number $n$, the decorations on $M_n$ very frequently have at least one mirror symmetry. Or, equivalently, the center of $M_n$ almost always lies on the boundary of the fundamental triangle $\Delta$.

\begin{figure}[h]
\centering
  \includegraphics[scale=0.6]{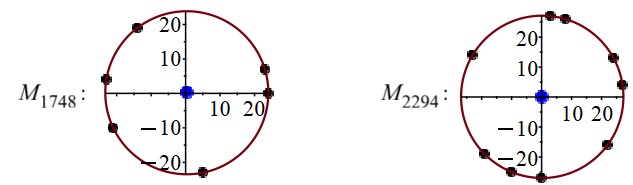}
\caption{Both 1748 and 2294 are strong MAC numbers whose decorated MAC circles do not have any mirror symmetry. However, their impacting indices are small: $I_{1748}=1$ and $I_{2294}=2$.}
\label{fig:2Exception}
\end{figure}
However, we have found two exceptions when $n\le 2700$ shown in Figure~\ref{fig:2Exception}. This disproves Conjecture 3(13) in \cite{Zhao2025Aug}, seemingly to show that 13 is really an unlucky number:)

We remark in passing that sometimes a MAC number is not strong but its largest decorated circle may still have a lot of decorations, as the cases shown in Figure~\ref{fig:MACsymAbundant}.

\begin{figure}[h]
\centering
  \includegraphics[scale=0.5]{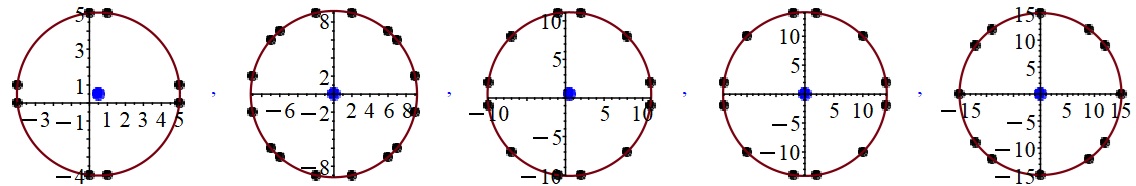}
\caption{The MAC numbers 60, 261, 348, 621, 697 are not strong, even though their decorated MAC circles shown here have more than one symmetry.}
\label{fig:MACsymAbundant}
\end{figure}

Notice that the decorations in Figure~\ref{fig:MACsymAbundant} are equipped with three kinds of mirror symmetries: horizontal, vertical, and slant, which is equivalent to the fact that the center of all $M_n$ in Figure~\ref{fig:MACsymAbundant} is at $(1/2, 1/2)$. The general relation between  the decoration symmetries of $M_n$ and the center of $M_n$ is given in Table~\ref{tbl:symmetryTypes}.

\begin{table}[h]
\caption{Relations between decoration symmetries of $M_n$ and the location of its center, and distribution of $M_n$-centers ($n\le 2700$).}
\begin{center}
\begin{tabular}{|c|c|c|}
\hline
Center Location of $M_n$ & Mirror Symmetries of $M_n$  & $\sharp\left\{n \left| {\textstyle M_n\text{-center=}\atop \textstyle \text{1st column} }\right.\right\}$ \\
\hline
$(0,0)$ or $(1/2,1/2)$ & horizontal, vertical and slant (45$^\circ$) &189 \\
\hline
$(1/2,0)$ & horizontal and vertical  &51\\
\hline
$(1/2,y),\ 0<y<1/2$ &  vertical &255\\
\hline
$(x,0),\ 0<x<1/2$ & horizontal  &337\\
\hline
$(x,x),\ 0<x<1/2$ & slant (45$^\circ$)   &477\\
\hline
$(x,y),\ 0<x\ne y<1/2$ & no symmetry & 1441\\
\hline
\end{tabular}
\end{center}
\label{tbl:symmetryTypes}
\end{table}

After computing the centers of the MAC circles $M_n$ for all MAC number $n\le 2700$, we plot their distribution on the fundamental triangle $\Delta$ in Figure~\ref{fig:centerPlot}.

\begin{figure}[h]
\centering
  \includegraphics[scale=0.25]{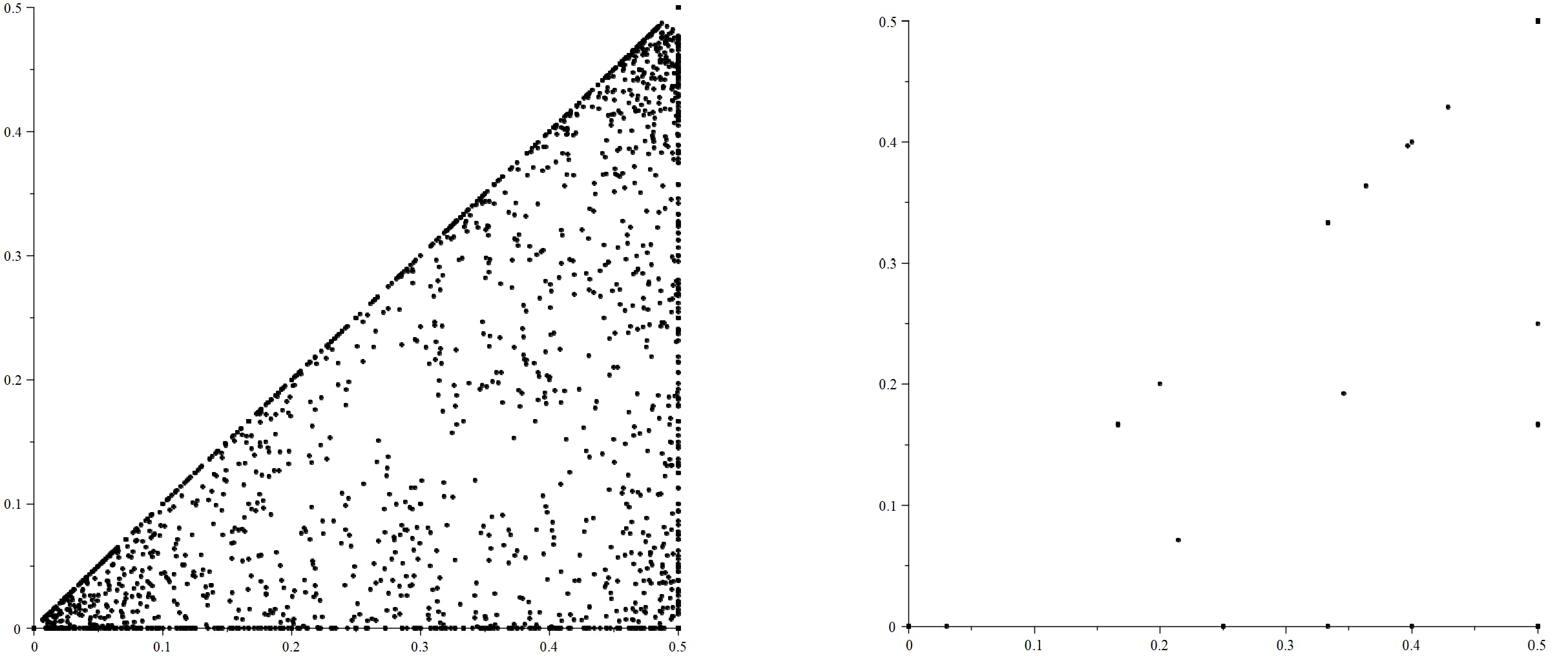}
\caption{Left: The picture shows the distribution of MAC circle centers in the fundamental triangle $\Delta$ for $0\le n\le 2700$. Some points have very high multiplicities: $(0,0)$ repeats 94 times, $(1/2,1/2)$ 95 times, and $(1/2,0)$ 51 times. 
Right: The picture shows the distribution of MAC circle centers in the fundamental triangle $\Delta$ for \textbf{strong} MAC $n$, $0\le n\le 2700$, with $(0,0)$ repeating 52 times, $(1/2,1/2)$ 69 times, and $(1/2,0)$ 21 times. The two points not on the boundary corresponds to two circles in Figure~\ref{fig:2Exception}.}
\label{fig:centerPlot}
\end{figure}

The high multiplicities at $(0,0)$, $(1/2,1/2)$ and $(1/2,0)$ reflect the fact that the mirror symmetries tend to make the corresponding circle $M_n$ rigid which can be quantified by $I_n$. In fact, about half of these multiplicities are contributed by the strong MAC numbers. Moreover, this is also related to the subtle but apparent empty area around the three points. It seems that these points act as attractors to nearby potential points. Another fact which cannot be seen from the above picture is that if two MAC numbers are close then their center locations are often close.

\subsection{Infinitude of strong MAC numbers and Pythagorean triples.}
We observe that Theorem \ref{thm:infinityMAC} implies that Conjecture~\ref{conj:nonMAC} concerning infinitude of non-MAC numbers is equivalent to the following statement.

\begin{conjecture}\label{conj:StrongMAC}
There are infinitely many strong MAC numbers.
\end{conjecture}

From numerical evidence, we further propose a much stronger statement as follows. Define
\begin{equation*}
{\mathcal P}:=\left\{
R\in {\mathbb N} \left |
\aligned
& \text{every prime factor $p$ of $R$ is 2} \\
& \text{ or satisfies } p\equiv 1 \pmod{4}
\endaligned \right.
\right\}.
\end{equation*}

\begin{conjecture}\label{conj:PythagorasStrongMAC}
Suppose the circle centered at the origin with radius $R$ is an $n_R$-circle. Then the subset of ${\mathcal P}$
\begin{equation*}
{\mathcal P}_{\text{strong}}:=\Big\{R\in {\mathcal P} \mid \text{$n_R$ is a strong MAC number}  \Big\}
\end{equation*}
has positive density in ${\mathcal P}$.
\end{conjecture}

Recall that $(a,b,c)\in{\mathbb N}^3$ is called a primitive Pythagorean triple if $\gcd(a,b,c)=1$ and $a^2+b^2=c^2$.
It is well known (cf. \cite{Silverman,Zagier1990}) that such a triple $(a,b,c)$ exists if and only if $c=R=p_1^{e_1}\cdots p_k^{e_k}$
where every $p_j$ is a prime satisfying $p_j\in {\mathcal P}$ for all $j=1,\dots,k$. It is not hard to see that there are exactly
$(2e_1+1)\cdots(2e_k+1)-1$ positive integer solutions $(x,y)$ using the well-known idea that if $(a_j,b_j,c_j)$ ($j=1,2$, $c_1\ne c_2$) are two primitive Pythagorean triples then $(|a_1b_1-a_2b_2|,a_1b_2+a_2b_1,c_1c_2)$ and  $(|a_1b_2-a_2b_1|,a_1b_1+a_2b_2,c_1c_2)$ together with their $x$- and $y$- coordinates swapped are four primitive Pythagorean triples. Hence, there are $4(2e_1+1)\cdots(2e_k+1)$ lattice points on the circle $M_n$ centered at $(0,0)$ of radius $R$, with the extra four points given by $(0,\pm R)$ and $(\pm R,0)$ not produced from the Pythagorean triples. Moreover, the decorated circle $M_n$ has all the three symmetries in Table~\ref{tbl:PythaRadius}. For all $n\le 2700$, the following eight strong MAC numbers are all produced in this way:
\begin{table}[h]
\caption{Strong MAC numbers produced by Pythagorean triples. Note that $10=2\cdot 5$, $20=4\cdot 5$, and $26=2\cdot 13$.}
\centering\begin{tabular}{|c|c|c|c|c|c|c|c|c|}
\hline
$R$ &   5 & 10 & 13 & 17 & 20   & 25   & 26 &29\\
\hline
$n$ & 69 & 305& 517 & 889 &1245 &1941 &2109 &2617 \\
\hline
$\sharp\{$pts on $M_n\}$ &   12&   12&   12&   12&   12&   20 &   12 &   12 \\
\hline
\end{tabular}
\label{tbl:PythaRadius}
\end{table}

However, we find that $697=n_{15}$ is MAC but not strong even though its MAC circle is centered at (0,0) and has radius $15=3\cdot 5$. See the last decorated circle in Figure~\ref{fig:MACsymAbundant}. We wonder if Conjecture~\ref{conj:PythagorasStrongMAC}  can actually be strengthened to 100\% instead of just positive density, or even further, ${\mathcal P}_{\text{strong}}={\mathcal P}$.

\begin{problem}
Suppose $R\in{\mathcal P}$ and the circle centered at the origin with radius $R$ is an $n_R$-circle. Is $n_R$ always a strong MAC number?
\end{problem}

\subsection{Distribution of strong MAC numbers.}
It is conjectured in \cite[Conjecture 3, (8)]{Zhao2025Aug} that the density of strong MAC numbers in the set of MAC numbers is positive. This is corroborated by the distribution data in the range $0\le n\le 2700$ in Table~\ref{tbl:strongMACdistrution}. Hence, our guess is that the density of strong MAC numbers in ${\mathbb N}_0$ should be positive at around 6\%. We point out that $M_{2700}$ has only four decorations and therefore it cannot be strong by Theorem~\ref{thm:DecorStrongMAC}.

\begin{table}[h]
\caption{Strong MAC number distribution for $0\le n\le 2700$.}
\centering\begin{tabular}{|c|c|c|c|c|c|c|c|c|c|}
\hline
$\lfloor n/300 \rfloor$  &   0 &  1  & 2  & 3 & 4 & 5 & 6 & 7 & 8 \\
\hline
$\sharp\{n|$  $n$ is strong MAC$\}$ & 26& 21& 20& 21& 11& 17& 16& 18& 17  \\
\hline
\end{tabular}
\label{tbl:strongMACdistrution}
\end{table}

\section{Non-MAC numbers.}\label{sec:non-MAC}
From the data collected for all $n\le 2700$ we have found that the non-MAC numbers are scattered pretty evenly in the set of positive integers.
It is conjectured in \cite[Conjecture 3, (5) and (7)]{Zhao2025Aug} that there are arbitrarily long sequence of non-MAC numbers and the density of non-MAC numbers in ${\mathbb N}_0$ is positive (perhaps around 15\%). The first part of this conjecture would follow from Conjecture~\ref{conj:infiniteI_n} and the second part is corroborated by the distribution data in Table~\ref{tbl:nonMACdistrution}.

\begin{table}[h]
\caption{Non-MAC number distribution for $n\le 2700$.}
\centering\begin{tabular}{|c|c|c|c|c|c|c|c|c|c|}
\hline
$\lfloor n/300 \rfloor$  &   0 &  1  & 2  & 3 & 4 & 5 & 6 & 7 & 8 \\
\hline
$\sharp\{n|$  $n$ is non-MAC$\}$ & 57& 49& 54& 47& 50& 40& 45& 44& 47  \\
\hline
\end{tabular}
\label{tbl:nonMACdistrution}
\end{table}

An interesting phenomenon associated with a non-MAC number $n$ is that the largest lattice $n$-circle may have a different feature from that of the MAC circles of MAC numbers, which is useful for our search of MAC numbers.

\begin{theorem}\label{thm:nonMACdecor}
If none of the triangles formed by the decorations of any largest lattice $n$-circle is acute, then $n$ must be a non-MAC number.
\end{theorem}
\begin{proof}
Let $M_n$ be a largest lattice $n$-circle. Suppose none of the triangles formed by any three decorations of $M_n$ is acute. We claim that if arc $\alpha$ on $M_n$ is the largest arc without other decorations in the interior then its angle measure is at least $180^\circ$. Suppose this not true, namely, the angle measure $m(\widearc{AB})\le 180^\circ$ where $A$ and $B$ are two decorations on $M_n$. Let $A'$ and $B'$ be the antipode of $A$ and $B$, respectively.
\begin{figure}[h]
\centering
  \includegraphics[scale=0.5]{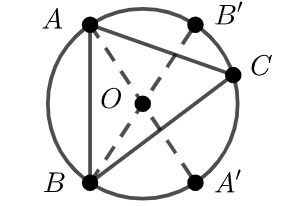}
\caption{Triangle $\triangle ABC$ is an acute or right triangle if $\widearc{AB}$ is a minor arc, where $O$ is the circle center.}
\label{fig:minorArc}
\end{figure}
Then there must be a decoration $C$ (could be equal to $A'$ or $B'$) on the minor arc $\widearc{A'B'}$ by maximality of $\widearc{AB}$. See Figure~\ref{fig:minorArc}. By the Inscribed Angle Theorem, $\angle C=m(\widearc{AB})/2\le 90^\circ$, $\angle A=m(\widearc{BA'C})/2\le m(\widearc{BA'B'})/2= 90^\circ$, and $\angle B=m(\widearc{AB'C})/2\le m(\widearc{AB'A'})/2= 90^\circ$. Hence, $\triangle ABC$ is an acute or right triangle contradicting to the assumption of the theorem.

Knowing that the largest arc $\alpha=\widearc{AB}$ without other decorations in the interior has angle measure at least $180^\circ$, we can push the circle by an infinitesimal amount towards this arc while keeping the two decorations $A$ and $B$ on the circle. Then we obtain a circle with the same number of interior points but with larger radius. This is absurd if $n$ is MAC since $M_n$ is a largest $n$-circle by Lemma~\ref{lem:latticeCircle}. Therefore, $n$ must be a non-MAC number.
\end{proof}

Theorem~\ref{thm:nonMACdecor} can expedite our computer search of the MAC numbers because we only need to search through all non-obtuse triangles centered in the first unit square with one of its vertices $(x,y)$ satisfying $x\ge 0$ and $y\ge x$. Further, we can use Lemma~\ref{lem:Rbound} to significantly reduce the number of vertices we need to search through.
To illustrate this idea, we can take a look at the case $n=336$. Our computation shows that $R_{332}=\sqrt{442}/2$, $\rho_{333},\dots,\rho_{338}<\sqrt{442}/2$. However, the above more efficient search only shows that $\rho_{336}'= 5\sqrt{685610}/398<R_{332}$ corresponding to the acute triangle with vertices at $(8, -3), (7,5), (8, 3)$. In fact, the largest lattice 336-circle centered on $\Delta$ has radius $\rho_{336}=5\sqrt{157}/6<R_{332}$ with decorations given in Figure~\ref{fig:336Decor}. We see that all the decorations stay on the bottom semicircle. Then Theorem~\ref{thm:nonMACdecor} implies that 336 in a non-MAC number.

\begin{figure}[h]
\centering
  \includegraphics[scale=0.5]{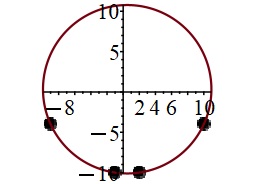}
\caption{This largest lattice 336-circle implies that 336 must be a non-MAC number.}
\label{fig:336Decor}
\end{figure}

\section{Higher dimensional analog.}\label{sec:higherDim}
In this last section, we explore briefly the higher dimensional extensions of the main results of the paper. To begin with, we fix dimension $d\ge 2$ and say a non-negative integer $n$ is \emph{$d$-maximally circlable} ($d$-MAC) if there is a largest $(d-1)$-dimensional sphere enclosing exactly $n$ lattice points in its interior in the $d$-dimensional coordinate space ${\mathbb R}^d$. Otherwise, we say $n$ is non-$d$-MAC. For example, maximally $2$-circlable means the same as maximally circlable defined in the main body of this paper. An \emph{$n$-enclosing sphere in} ${\mathbb R}^d$ is a $(d-1)$-dimensional sphere that enclosing exactly $n$ lattice points in its interior. Denote by $R_n^{(d)}$ the largest radius of all $n$-enclosing spheres in ${\mathbb R}^d$ if such a radius exists and call it the \emph{$d$-MAC radius of $n$}; otherwise, denote by $R_n^{(d)}$ the lowest upper bound of radii (LUBOR as the acronym) of all $n$-enclosing $d$-dimensional spheres and call it the \emph{$d$-LUBOR of $n$}.

We call a sphere in ${\mathbb R}^d$ a lattice sphere if there are at least $d+1$ lattice points on its boundary.
For any $d$-MAC number $n$, let $M_n^{(d)}$ be a $n$-enclosing sphere in ${\mathbb R}^d$ whose radius is $R_n^{(d)}$, which we call a \emph{$d$-MAC sphere of $n$}. We call the lattice points on the boundary of $M_n^{(d)}$ \emph{decorations}. The next result, which generalizes Lemma~\ref{lem:latticeCircle}, can be proved readily by the perturbation method.

\begin{lemma}\label{lem:Lattice-d-MAC}
If $n$ is a $d$-MAC number, then $M_n^{(d)}$ must be a lattice sphere. Namely, $M_n^{(d)}$ has at least $d+1$ decorations.
\end{lemma}

By essentially the same argument as in the proof of Lemma \ref{lem:main}, one can also prove the following result.
\begin{lemma}\label{lem:main-d}
Let $d$ be any positive integer with $d\ge 2$. Let $\{R_n^{(d)}: n\ge 0\}$ be the positive sequence of $d$-MAC radii (for $d$-MAC integers $n$)
and $d$-LUBORs (for non-$d$-MAC integers $n$). Then we have:
\begin{enumerate}
  \item[(a)] $\{R_n^{(d)}: n\ge 0\}$ is a non-decreasing sequence.

  \item[(b)] Suppose $k<\ell$ are two consecutive $d$-MAC integers and $n$ is non-$d$-MAC such that $k<n<\ell$. Then $R_n^{(d)}=R_k^{(d)}$, namely, $R_k^{(d)}$ is the $d$-LUBOR for all such non-$d$-MAC $n$ immediately following $k$.

  \item[(c)] Suppose $k<n$, $k$ is a $d$-MAC integer and all the integers between $k$ and $n$ (exclusive) are non-$d$-MAC. Then the radius of the largest lattice $n$-enclosing sphere is \textbf{strictly} smaller than $R_k^{(d)}$ if and only if $n$ is non-$d$-MAC.
\end{enumerate}
\end{lemma}

Using this lemma one can easily extend Theorem~\ref{thm:infinityMAC} to the following by essentially the same argument.

\begin{theorem}\label{thm:infinity-d-MAC}
There are infinitely many $d$-MAC numbers.
\end{theorem}

Further, we can define the $d$-impacting index $I_n^{(d)}$ of a $d$-MAC number $n$ to be the number of consecutive non-$d$-MAC numbers following $n$. If $n$ is non-$d$-MAC then we set $I_n^{(d)}=-1.$ A $d$-MAC number $n$ is called \emph{strong} if $I_n^{(d)}>0$. The following are all the strong $3$-MAC numbers for $n\le 200$:

2, 4, 8, 16, 20, 24, 28, 30, 33, 40, 48, 52, 56, 68, 70, 72, 75, 78, 81, 84, 88, 93, 106, 112, 114, 118, 128, 136, 140, 145, 148, 152, 160, 179, 194, 196.

Each of these strong numbers $n$ has a unique $3$-MAC sphere $M_n^{(3)}$. About half of these numbers have their $d$-impacting index equal to 1. When $I_n^{(3)}>1$ we list them in Table \ref{tbl:strong-d-MAC}.

\begin{table}[h]
\caption{Distribution of strong $3$-MAC numbers $n$ with their impacting index $>1$, $0\le n\le 200$. $D_n^{(3)}$ is the number of decorations on the unique $M_n^{(3)}$, satisfying $D_n^{(3)}\ge 2I_n^{(3)}+4$.}
\centering\begin{tabular}{|c|c|c|c|c|c|c|c|c|c|c|}
\hline
$n$  & 8& 20& 24& 33& 40& 56& 72& 81& 88& 93 \\
\hline
$I_n^{(3)}$ & 7& 3& 2& 6& 3& 11& 2& 2& 4& 10\\
\hline
$D_n^{(3)}$ &24 & 16& 8& 24& 13& 32& 12& 12& 48& 30\\
\hline
$n$  &118& 128& 136& 140& 148& 152& 160& 179& 196 &\\
\hline
$I_n^{(3)}$&   9& 3& 3& 2& 2& 2& 18& 4& 3 &\\
\hline
$D_n^{(3)}$& 32& 16& 24& 9& 12& 16& 48& 24& 16 &\\
\hline
\end{tabular}
\label{tbl:strong-d-MAC}
\end{table}

The following result is a generalization of Theorem~\ref{thm:DecorStrongMAC}.

\begin{theorem}\label{thm:DecorStrong-d-MAC}
For each strong $d$-MAC number $n$ with $d$-impacting index $I_n^{(d)}$, its $d$-MAC sphere $M_n^{(d)}$ has at least $2I_n^{(d)}+d+1$ decorations.
\end{theorem}

\begin{proof}
The proof is similar to that of Theorem~\ref{thm:DecorStrongMAC}. For simplicity, let's drop the superscript $(d)$ in the rest of the proof.

Suppose there are at most $2I_n+d$ decorations on $M_n$. From Lemma~\ref{lem:Lattice-d-MAC}, we know there are at least $d+1$ decorations on $M_n$. Thus, there is a hyperplane $\mathcal P$ that passes through $d-1$ decorations. Then, on one side of $\mathcal P$ there are at most $I_n$ decorations not lying on $\mathcal P$. By moving the center of the sphere towards these points along the direction perpendicular to $\mathcal P$ with an infinitesimal distance we see that there is a sphere of radius $R_n$ that encloses exactly $n+I_n$ points. By discreteness of lattice points, we can further increase the radius of this circle slightly while keeping the number of interior points unchanged. This contradicts to the condition that $n+I_n$ is non-$d$-MAC and that $R_n$ is its LUBOR by Lemma\ref{lem:main-d}(b). This contradiction implies that there are at least $2I_n+d+1$ decorations on $M_n$.
\end{proof}

Note that Theorem~\ref{thm:DecorStrong-d-MAC} is tight when $d=3$ because we see that $I_{16}^{(3)}=1$ while there are exactly $D_{16}^{(3)}=6$ decorations on its $3$-MAC sphere $M_{16}^{(3)}$ centered at $(1/3,1/3,0)$ with radius $\sqrt{26}/3$. The bound is also reached with $n=24$ since $D_{24}^{(3)}=8$ and $I_{24}^{(3)}=2$.

Unlike the 2-dimension situation where we found two 2-MAC numbers $n$ (i.e. $n=322$ and $n=941$) each of which has two lattice non-equivalent MAC $n$-circles, we have not found any 3-MAC number $n$ with two lattice non-equivalent MAC $n$-spheres.

\begin{problem}\label{prob:3-decor}
Is there a 3-MAC number $n$ which has two lattice non-equivalent MAC $n$-spheres?
\end{problem}

Although $30$, 40 and 84 are all strong $3$-MAC numbers, there are 9 decorations on $M_{30}^{(3)}$, 13 decorations on $M_{40}^{(3)}$, and 9 decorations on $M_{84}^{(3)}$. This is in contrast to the situation in ${\mathbb R}^2$, where Conjecture~\ref{conj:DecorStrongMACEven} claims that the number of decorations for MAC circles should always be even.

From the numerical evidence we propose the following conjectures for general $d\ge 2$.
\begin{conjecture}\label{conj:DecorStrong-d-MAC}
Let $d\ge2$. Then
\begin{enumerate}
  \item [(a)] For each $n\in\mathbb N$, there is an $n$-enclosing $(d-1)$-dimensional sphere of volume $n$.

  \item [(b)] There are infinitely many non-$d$-MAC numbers. Or, equivalently, there are infinitely many strong $d$-MAC numbers.

  \item [(c)] The set of $d$-impacting indices $\{I_n^{(d)}:n\in {\mathbb N}\}$ is infinite.
\end{enumerate}
\end{conjecture}

Note that the volume of an $d$-dimensional sphere with radius $r$ is given by $\pi^{d/2} r^d/\Gamma(d/2 + 1)$, where $\Gamma$ is Euler's Gamma function.
\begin{theorem}
For each $n\in\mathbb N$, the first inequality of the following
\begin{equation*}
     \sqrt[d]{n\Gamma(d/2 + 1)}/\sqrt{\pi}< R_n^{(d)} <\sqrt[d]{n\Gamma(d/2 + 1)}/\sqrt{\pi}+\sqrt{d}
\end{equation*}
holds if Conjecture~\ref{conj:DecorStrong-d-MAC}(a) is true, and the second inequality always holds.
\end{theorem}
\begin{proof}
We only need to show the second inequality. This can be achieved by an argument similar to the proof of \cite[Theorem 9]{Zhao2025Aug}.
\end{proof}

\medskip
\noindent
{\bf Acknowledgments.}
The author thanks the support of the Jacobs Prize from The Bishop's School. He is also grateful for the immense help provided by Dr. Marcus Jaiclin at The Bishop's School with the numerical computation.

\begin{figure}[h]
\centering
  \includegraphics[scale=0.4]{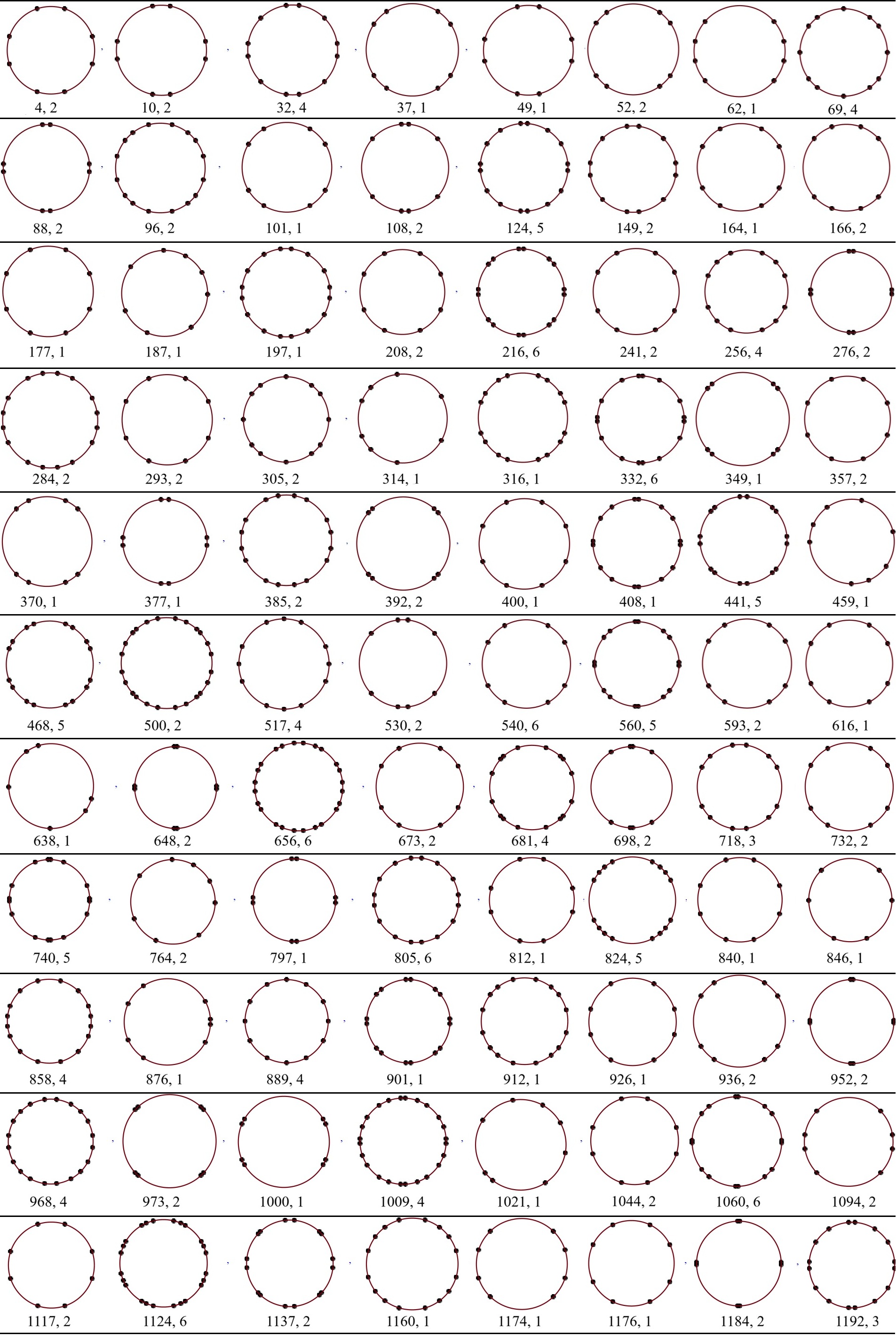}
\end{figure}

\begin{figure}[h]
\centering
  \includegraphics[scale=0.4]{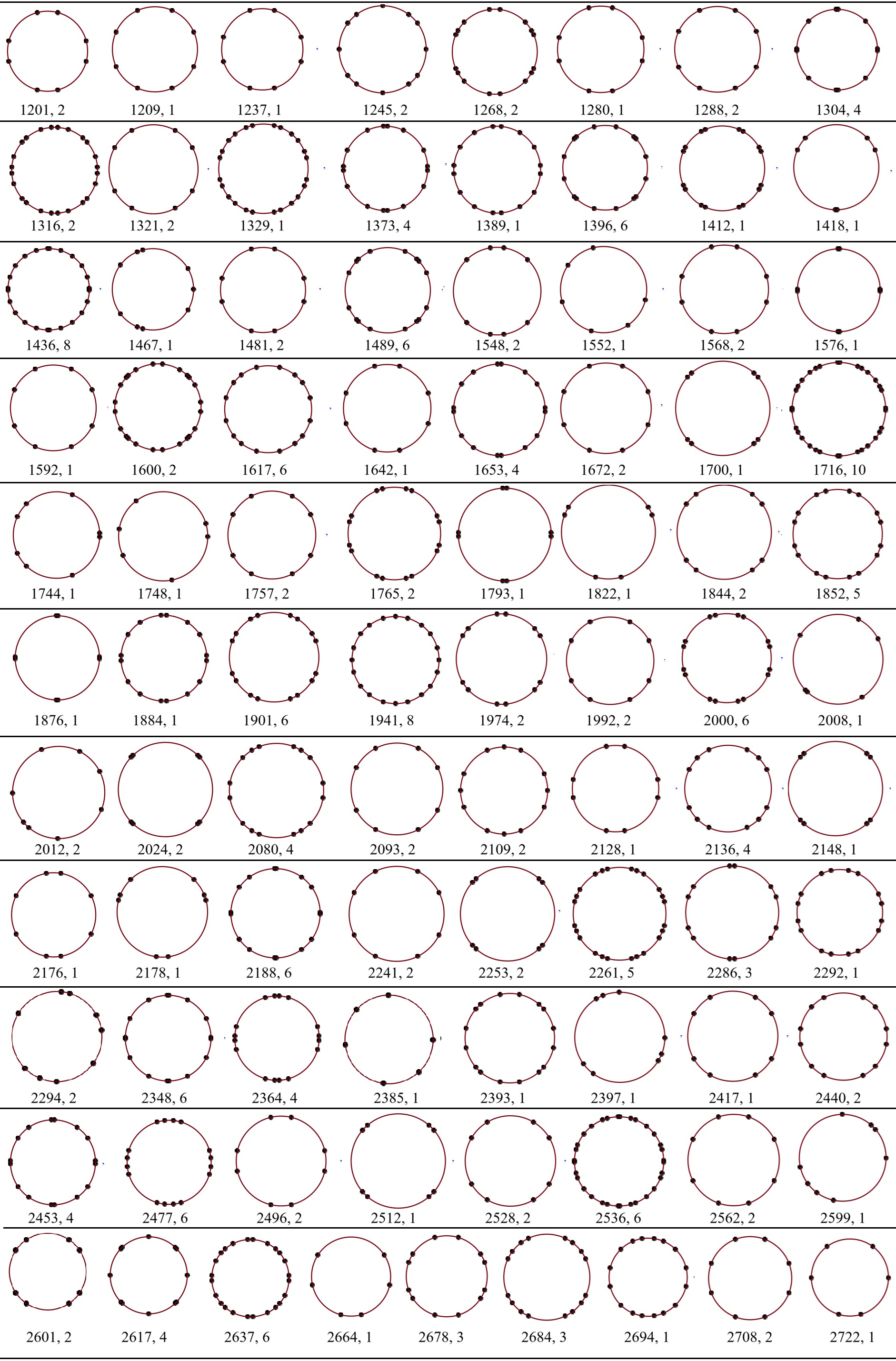}
\caption{All strong MAC numbers with their decorated MAC circles $M_n$ and impacting indices.}
\label{fig:AllMACcircle2a}
\end{figure}

\end{document}